\documentclass[11pt]{amsart}

\usepackage[utf8]{inputenc}
\usepackage[T1]{fontenc}
\usepackage[english]{babel}
\usepackage{amsmath,amsthm,amssymb,mathtools}
\usepackage{graphicx}
\usepackage{booktabs,array}
\usepackage{geometry}
\usepackage[justification=raggedright,singlelinecheck=false]{caption}
\usepackage[
  bookmarksdepth=3,
  pdfdisplaydoctitle=true,
  pdfusetitle,
  colorlinks=true,
  urlcolor=blue,
  linkcolor=blue,
  citecolor=blue
]{hyperref}

\newtheorem{theorem}{Theorem}[section]
\newtheorem{lemma}[theorem]{Lemma}

\theoremstyle{definition}

\newtheorem{remark}[theorem]{Remark}

\numberwithin{equation}{section}

\newcommand{\lapl}{\Delta}
\newcommand{\grad}{\nabla}

\renewcommand{\S}{\mathbb{S}}
\renewcommand{\H}{\mathbb{H}}

\newcommand{\R}{\mathbb{R}}

\newcommand{\D}{\mathbb{D}}

\newcommand{\si}{\operatorname{si}}
\newcommand{\ta}{\operatorname{ta}}

\title{Scaling Inequalities and Limits for Clamped Plate Eigenvalues on Geodesic Disks}
\author{Scott Harman}
\date{}

\begin{document}

\begin{abstract}
For the bilaplacian in spherical and hyperbolic spaces, clamped vibrating and buckling eigenvalues on geodesic disks are shown to satisfy scaling inequalities analogous to the standard scale invariance of Euclidean plate eigenvalues. These results extend curved-space scaling results from the Laplacian to fourth-order eigenvalue problems. In addition, the limiting behavior of the scaled eigenvalues is determined as the disks expand to fill the ambient space. Depending on the geometry and the choice of scaling, the eigenvalues tend to zero, positive or negative infinity, or finite values associated with the bottom of the hyperbolic spectrum. The negative divergence in one spherical buckling limit is related by conformal inversion to an exterior buckling problem.
\end{abstract}

\maketitle

\section{Introduction and results}
\subsection*{Curving the free plate}
We consider the fourth-order bilaplacian operator. Throughout we use the analyst's sign convention, so that $-\lapl$ is a nonnegative operator. Let $M$ be a Riemannian manifold and $\Omega \subset M$ be a bounded domain with a smooth boundary. In this paper, we are interested in the following fourth-order eigenvalue equations
\begin{align}
&\lapl^2 u - \tau \lapl u = \Gamma u \hspace{1.1cm} \text{in } \Omega,\label{eq:vibreigenprob} \\
&\lapl^2 u + \rho u = \Lambda (-\lapl) u \quad\text{in } \Omega,\label{eq:buckeigenprob}
\end{align}
where $\rho$ and $\tau$ are real-valued parameters, $\Gamma$ and $\Lambda$ are the eigenvalues, and $\lapl$ is the Laplace--Beltrami operator on $M$ (with the analyst's convention that $-\lapl$ is a positive operator). Both equations arise naturally in the linearized theory of elastic plates. We will impose appropriate boundary conditions (specified below). The first equation arises when considering the frequencies of a vibrating plate with a tensile parameter $\tau$. The plate is subject to tension when $\tau > 0$ and compression when $\tau < 0$. The second equation arises when considering the buckling load of a plate with elasticity parameter $\rho$. The plate is subject to a restorative force when $\rho > 0$ and a destabilizing force when $\rho < 0$. More information can be found in \cite{K93}.

Write $\Gamma(\Omega, \tau)$ for the \textit{vibrating eigenvalue} in Equation \eqref{eq:vibreigenprob} and $\Lambda(\Omega, \rho)$ for the \textit{buckling eigenvalue} in Equation \eqref{eq:buckeigenprob}. When our manifold $M$ is Euclidean space, the eigenvalues satisfy the following scaling identities for all $t>0$:
\begin{align}
t^4\Gamma(t\Omega, \tau/t^2) &= \Gamma(\Omega, \tau), \label{eq:vibratlaw} \\
t^2\Lambda(t\Omega, \rho/t^4) &= \Lambda(\Omega, \rho). \label{eq:bucklaw}
\end{align}

Such identities fail in curved spaces in general, even when suitably adapting linear dilation to manifolds (e.g.\ geodesic discs). In this paper, we uncover analogous scaling information for these plate eigenvalues on manifolds.

\subsection*{Looking at caps and disks}

We use the coordinate framework from \cite{H25}. Our attention will be on spherical caps in spherical space and geodesic disks in hyperbolic space. For the sphere $\S^2$, we let $(\theta, \xi)$ be our coordinate system where $\theta \in [0,
\pi]$ is the angle of aperture from the north pole and $\xi \in [0, 2\pi)$ represents the longitudinal angle around the equator. For hyperbolic space $\H^2$, we employ the analogous coordinate system where $\theta > 0$ represents geodesic distance from the origin and $\xi \in [0, 2\pi)$. For $\Theta>0$, we define the spherical cap of aperture $\Theta$ by
\[
C(\Theta):=\Bigl\{\,p\in \mathbb{S}^2:\ \text{in coordinates }p=(\theta,\xi),
\ 0\le \theta<\Theta,\ \xi\in[0,2\pi)\Bigr\}.
\]
Similarly, for $\Theta<0$, we define the hyperbolic disk of geodesic radius $|\Theta|$ by
\[
C(\Theta):=\Bigl\{\,p\in \mathbb{H}^2:\ \text{in coordinates }p=(\theta,\xi),
\ 0\le \theta<|\Theta|,\ \xi\in[0,2\pi)\Bigr\}.
\]

When needed, we will denote their associated Laplace-Beltrami operators as $\lapl_{\S^2}$ or $\lapl_{\H^2}$, but when the operator is clear from context we will simply write $\lapl$. We denote the vibrating eigenvalue on $C(\Theta)$ as $\Gamma(\Theta, \tau)$ and the buckling eigenvalue as $\Lambda(\Theta, \rho)$.

Speaking informally, when $\Theta$ is close to $0$ the cap $C(\Theta)$ is nearly Euclidean at the relevant scale, so any reasonable scaling for plate eigenvalues should recover the Euclidean identities \eqref{eq:vibratlaw}--\eqref{eq:bucklaw} in the limit $\Theta\to 0$. Motivated by \eqref{eq:vibratlaw}, we plan to multiply the vibrating eigenvalues by a scaling factor that behaves like $\Theta^{4}$ near $\Theta=0$ and to normalize $\tau$ by a factor that behaves like $\Theta^{2}$ near $\Theta=0$. For the buckling eigenvalues, in accordance with \eqref{eq:bucklaw}, we instead want our scaling factor to behave like $\Theta^{2}$ and to normalize $\rho$ by a factor that behaves like $\Theta^{4}$ near $\Theta=0$.

Accordingly, we introduce two auxiliary functions $\mu(\Theta)$ and $\nu(\Theta)$, each behaving like $\Theta$ as $\Theta\to 0$, whose powers will produce the necessary scaling and normalization factors. Concretely, $\mu(\Theta)^4$ will serve as the scaling factor for the vibrating problem and $\mu(\Theta)^2$ as the scaling factor for the buckling problem, while $\nu(\Theta)^2$ and $\nu(\Theta)^4$ will serve as the corresponding normalizing factors for $\tau$ and $\rho$, respectively.

Now we describe our choices for $\mu$ and $\nu$. We define the functions

\[
\si(\Theta) = 
\begin{cases} 
\sinh(|\Theta|) & \text{for } \Theta < 0, \\
\sin(\Theta) & \text{for } \Theta \geq 0,
\end{cases} \qquad \ta(\Theta) = 
\begin{cases} 
\tanh(|\Theta|) & \text{for } \Theta < 0, \\
\tan(\Theta) & \text{for } \Theta \geq 0.
\end{cases}
\]
Two natural geometric quantities on our curved spaces are $2\ta(\Theta/2)$ and $\si(\Theta)$. These functions behave like $\Theta$ around $\Theta = 0$, and their second and fourth powers will serve as our scaling and normalizing factors. Indeed, second and fourth powers of either $\si(\Theta)$ or $2\ta(\Theta/2)$ act as scaling or normalizing factors in the theorems below, in various combinations.

\subsection*{Main results}

Subject to a choice of boundary conditions, Equations \eqref{eq:vibreigenprob} and \eqref{eq:buckeigenprob} generate discrete spectra of eigenvalues. There are several choices one may consider (see \cite{C11}), but our main focus will be on the clamped boundary
\[
u = |\grad u| = 0 \quad \text{on } \partial \Omega.
\]
On our geodesic disks $C(\Theta)$, in particular, we generate the two discrete spectra
\begin{align*}
\Gamma_1(\Theta, \tau) &\leq \Gamma_2(\Theta, \tau) \leq \Gamma_3(\Theta, \tau) \leq \dots , \\
\Lambda_1(\Theta, \rho) &\leq \Lambda_2(\Theta, \rho) \leq \Lambda_3(\Theta, \rho) \leq \dots.
\end{align*}
The first displayed spectrum represents the \textit{clamped vibrating eigenvalues} and the second represents the \textit{clamped buckling eigenvalues}; this will be our notation henceforth. When $\tau =0$ or $\rho = 0$ we recover the standard bilaplacian eigenvalues.

Define the functions
\[
S(\Theta) = \si(\Theta), \quad T(\Theta) = 2\ta(\Theta/2).
\]
One may interpret $S$ as the perimeter of the geodesic disk and $T$ as the stereographic radius. Let $\D\subset\R^2$ denote the Euclidean unit disk; whenever we refer to a limiting value at $\Theta=0$, it is to be interpreted as the corresponding Euclidean eigenvalue on the unit disk $\D$.

In each result, $\Theta$ runs over $(-\infty, \pi)$. We first state monotonicity results for the vibrating problem. 

\textbf{Note.}
Since $C(0)=\emptyset$, any occurrence below of a normalized eigenvalue at $\Theta=0$
is understood as a limit. For example,
\[
\lim_{\Theta\to 0}\Gamma_j(\Theta,\tau/S(\Theta)^2)\,S(\Theta)^4=\Gamma_j(\D,\tau),
\]
where $\Gamma_j(\D,\tau)$ denotes the $j$-th clamped plate eigenvalue of \eqref{eq:vibreigenprob} on the Euclidean
unit disk $\D\subset\R^2$ with parameter $\tau$. The same convention applies to the
other normalized quantities below, with $\Lambda_j$ in place of $\Gamma_j$ when appropriate. See Section \ref{sec:euc_limits_buck_vibr} for a proof.

\begin{theorem}[Vibrating monotonicity for nonnegative $\tau$]\label{th:vibratingmono}
Fix $j\ge 1$. If $\tau\ge 0$, then:
\begin{itemize}
\item $\Gamma_j(\Theta,\tau/S(\Theta)^2)\,S(\Theta)^4$ strictly decreases from $\infty$ to $0$,
\item $\Gamma_j(\Theta,\tau/T(\Theta)^2)\,T(\Theta)^4$ strictly increases from $1+\tau$ to $\infty$.
\end{itemize}
\end{theorem}

For the buckling eigenvalues with $\rho \leq 0$, the monotonicity is similar to the vibrating case.

\begin{theorem}[Buckling monotonicity for nonpositive $\rho$]\label{th:bucklingneg}
Fix $j\ge 1$. If $\rho\le 0$, then:
\begin{itemize}
\item $\Lambda_j(\Theta,\rho/S(\Theta)^4)\,S(\Theta)^2$ strictly decreases from $\infty$ to $-\infty$ when $\rho<0$, and from $\infty$ to $0$ when $\rho=0$.
\item $\Lambda_j(\Theta,\rho/T(\Theta)^4)\,T(\Theta)^2$ strictly increases from $1+\rho$ to $\infty$.
\end{itemize}
\end{theorem}

For nonnegative $\rho$, the monotonicity becomes mixed: the $S$- and $T$-normalizations in
Theorem~\ref{th:bucklingneg} swap places in Theorem~\ref{th:bucklingpos}.

\begin{theorem}[Buckling monotonicity for nonnegative $\rho$]\label{th:bucklingpos}
Fix $j\ge 1$. If $\rho\ge 0$, then:
\begin{itemize}
\item $\Lambda_j(\Theta,\rho/T(\Theta)^4)\,S(\Theta)^2$ strictly decreases from $\infty$ to $0$,
\item $\Lambda_j(\Theta,\rho/S(\Theta)^4)\,T(\Theta)^2$ strictly increases from $1$ to $\infty$.
\end{itemize}
\end{theorem}

\noindent \emph{Remark.}
The first item of Theorem~\ref{th:bucklingneg} is the only limiting regime considered here in which every fixed-index eigenvalue can diverge to $-\infty$. Its relation to an exterior buckling problem is discussed in Remark~\ref{rem:exterior-buckling}.

The special cases $\tau=0$ and $\rho=0$ are covered by these theorems and yield monotonicity for the standard
vibrating and buckling eigenvalues (i.e., without tensile or elasticity parameters).

\section{Preliminaries: boundary conditions and discrete spectrum}

Let $M$ be a Riemannian manifold and $\Omega \subset M$ be a bounded domain with a smooth boundary. Recall Equations \eqref{eq:vibreigenprob} and \eqref{eq:buckeigenprob}, which we restate here:
\begin{align*}
&\lapl^2 u - \tau \lapl u = \Gamma u \hspace{1.1cm} \text{in } \Omega, \\
&\lapl^2 u + \rho u = \Lambda (-\lapl) u \quad\text{in } \Omega.
\end{align*}

Appropriate boundary conditions are necessary to ensure discreteness of the
resulting spectra. In this paper we work in the clamped case, so the admissible
functions lie in $H_0^2(\Omega)$. Background material in
\cite{C11,CP22,K93,R77} discusses other boundary conditions that one may consider. 

To properly investigate these eigensystems, we provide a brief overview of some Riemannian geometry.

\subsection*{Sobolev spaces and the Hessian}

Most of the information provided here, along with further background material, can be found in \cite{H96}.

The Riemannian metric will be denoted locally by
\[
g=g_{ij}\,dx^i\,dx^j,
\]
with $g^{ij}$ representing components of the inverse metric.
For covariant $m$-tensors $A$ and $B$, we define their pointwise inner product by
\[
\langle A,B\rangle
:=
g^{i_1j_1}\cdots g^{i_mj_m}
A_{i_1\cdots i_m}B_{j_1\cdots j_m}.
\]
Let $\nabla^m u$ be the $m$-th covariant derivative of $u$. Our relevant operators will be the gradient $\grad$ and Hessian $\grad^2$ which have local coordinate representations
\[
(\nabla u)_i=\partial_i u,
\qquad
(\nabla^2 u)_{ij}=\partial_{ij}u-\Gamma_{ij}^k\partial_k u,
\]
where $\Gamma_{i,j}^k$ are the Christoffel symbols.
In particular for $m$-tensors $\grad^m u$ and $\grad^m v$, we have
\[
\langle \nabla^m u,\nabla^m v\rangle
:=
g^{i_1j_1}\cdots g^{i_mj_m}
(\nabla^m u)_{i_1\cdots i_m}(\nabla^m v)_{j_1\cdots j_m},
\]
and
\[
|\nabla^m u|^2=\langle \nabla^m u,\nabla^m u\rangle.
\]
The pointwise inner product allows us to define the Sobolev $H^k(\Omega)$-inner product
\[
\langle u,v\rangle_{H^k(\Omega)}
:=
\sum_{m=0}^k \int_\Omega \langle \nabla^m u,\nabla^m v\rangle\,dV,
\]
and corresponding norm
\[
\|u\|_{H^k(\Omega)}
:=
\left(\sum_{m=0}^k \int_\Omega |\nabla^m u|^2\,dV\right)^{1/2}.
\]
For $k \geq 1$, we define $H^k(\Omega)$
to be the Sobolev space of functions with covariant derivatives up to order $k$ and finite $H^k$ norm. We further define
\[
H_0^k(\Omega)
:=
\overline{C_c^\infty(\Omega)}
\]
where the closure is with respect to the $H^k$ norm.
The set $H_0^2(\Omega)$ will serve as the natural function space for our eigensystems.

\subsection*{Bilinear forms}

Let $u,v \in H_0^2(\Omega)$. The bilinear forms associated to \eqref{eq:vibreigenprob} and \eqref{eq:buckeigenprob} are, respectively,
\begin{align*}
a(u,v)
&:= \int_{\Omega} \langle \nabla^2 u,\nabla^2 v\rangle \, dV
   + \int_{\Omega} \mathrm{Ric}(\nabla u,\nabla v)\, dV
   + \tau\int_{\Omega} \langle \nabla u,\nabla v\rangle \, dV,
\\
b(u,v)
&:= \int_{\Omega} \langle \nabla^2 u,\nabla^2 v\rangle \, dV
   + \int_{\Omega} \mathrm{Ric}(\nabla u,\nabla v)\, dV
   + \rho\int_{\Omega} u v \, dV.
\end{align*}
where $\text{Ric}$ is the Ricci curvature tensor.
The corresponding weak formulations read
\begin{align*}
a(u,v) &= \Gamma\int_{\Omega} u v \, dV,
\\
b(u,v) &= \Lambda\int_{\Omega} \langle \nabla u,\nabla v\rangle \, dV.
\end{align*}

One might try to work with $\int_{\Omega} (\Delta u)(\Delta v)\,dV$ instead of Hessians in the bilinear forms, but boundary terms must be handled carefully (see \cite[Remark 3.3]{CP22}). We will employ a variation of Reilly's identity (see Reilly~\cite{R77} for the original). The identity from \cite[Eq. 3.11]{CP22} reads
\begin{equation}\label{eq:reillypolar}
\begin{split}
\int_\Omega (\Delta u)(\Delta v)\,dV
&=
\int_\Omega \langle \nabla^2 u,\nabla^2 v\rangle\,dV
+\int_\Omega \text{Ric}(\nabla u,\nabla v)\,dV \\
&\quad
+\int_{\partial\Omega}
\Bigl(
(n-1)\mathcal{H}\,\partial_\nu u\,\partial_\nu v
+\Delta_{\partial\Omega}u\,\partial_\nu v
+\partial_\nu u\,\Delta_{\partial\Omega}v
+II(\nabla_{\partial\Omega}u,\nabla_{\partial\Omega}v)
\Bigr)\,dS.
\end{split}
\end{equation}
Here, $\mathcal{H}$ and $II$ represent the shape operator and second fundamental form respectively, and $\partial_\nu$ and $\grad_{\partial \Omega}$ represent the normal and tangential derivatives respectively. The precise definition of $\mathcal{H}$ and $II$ can be found in \cite{CP22}. All we need though is that $u,v \in H_0^2(\Omega)$ force each term in the boundary integral in \eqref{eq:reillypolar} to be zero. As a result, for $u ,v\in H_0^2(\Omega)$ we have
\[
\int_\Omega (\Delta u)(\Delta v)\,dV
=
\int_\Omega \langle \nabla^2 u,\nabla^2 v\rangle\,dV
+\int_\Omega \text{Ric}(\nabla u,\nabla v)\,dV,
\]
and hence the bilinear forms simplify to
\begin{align}
a(u,v)
&= \int_{\Omega} \lapl u \lapl v \, dV
   + \tau\int_{\Omega} \langle \nabla u,\nabla v\rangle \, dV, \label{eq:correct_vib_bi_form}
\\
b(u,v)
&= \int_{\Omega} \lapl u \lapl v \, dV
   + \rho\int_{\Omega} u v \, dV. \label{eq:correct_buc_bi_form}
\end{align}

\subsection*{Function spaces}

We now specify the setup needed for the discrete spectral theorem. First, consider the vibrating problem. We take
\[
S_1 := H^2_0(\Omega),\qquad T_1 := L^2(\Omega),
\]
so that the embedding $\iota:S_1\hookrightarrow T_1$ is compact. To apply the spectral theorem to \eqref{eq:correct_vib_bi_form}, we need a bounded, bilinear form on $S_1$ which is coercive.

Since the Ricci tensor is bounded on $\overline{\Omega}$, we have
\[
\biggl|\int_{\Omega}\mathrm{Ric}(\nabla u,\nabla u)\,dV\biggr|
\le \|\mathrm{Ric}\|_{L^\infty(\Omega)}\,\|\nabla u\|_{L^2(\Omega)}^2.
\]
We control $\|\nabla u\|_{L^2}^2$ by $\|\nabla^2u\|_{L^2}^2$ up to an $L^2$-term via a Cauchy-with-$\varepsilon$ estimate:
\begin{equation}\label{eq:epsinterp}
\|\nabla u\|_{L^2(\Omega)}^2 \le \varepsilon \|\nabla^2 u\|_{L^2(\Omega)}^2 + C_\varepsilon \|u\|_{L^2(\Omega)}^2,
\end{equation}
valid on bounded smooth domains (see, e.g., \cite[Ch.~5]{E10} or \cite{L12}). Combining these bounds yields the inequality
\[
a(u,u)\ge \frac12\|\nabla^2u\|_{L^2(\Omega)}^2 - C\|u\|_{L^2(\Omega)}^2
\qquad\text{for all }u\in H^2_0(\Omega),
\]
with $C$ depending on $\|\mathrm{Ric}\|_{L^\infty(\Omega)}$ and $\tau$. Choosing $K>C$ and defining the shifted form
\[
a_K(u,v):=a(u,v)+K\int_{\Omega}uv\,dV,
\]
we obtain the estimate
\begin{equation}\label{eq:coerciveAK}
a_K(u,u) \ge c\,\|u\|_{H^2(\Omega)}^2, \qquad u\in H^2_0(\Omega),
\end{equation}
for some $c>0$, establishing coercivity of $a_K$. The eigenvalues for $a$ are obtained by subtracting $K$.

Now we look at the buckling case corresponding to the bilinear form \eqref{eq:correct_buc_bi_form}. Here, the compact embedding will be $\iota:S_2\to T_2$ where
$S_2$ is $H_0^2(\Omega)$ and $T_2$ is $H_0^1(\Omega)$. Note in this case that we are not embedding into $L^2(\Omega)$. The inner product
\[
(u,v)_{T_2}:=\int_\Omega \langle \nabla u,\nabla v\rangle\,dV,
\]
is equivalent to the standard $H^1$ norm via the Poincaré inequality
\[
\|u\|_{L^2(\Omega)} \leq C \|\nabla u\|_{L^2(\Omega)},
\qquad u\in H_0^1(\Omega).
\]
Coercivity now follows quickly. Define the shifted form
\[
b_K(u,v):=b(u,v)+K\int_\Omega \langle \nabla u,\nabla v\rangle\,dV.
\]
On $T_2$, Poincar\'e gives $\|u\|_{L^2}\le C\|\nabla u\|_{L^2}$, so the term
$\rho\int_\Omega u^2\,dV$ is controlled by $\|\nabla u\|_{L^2}^2$, and
\[
\left|\int_\Omega \mathrm{Ric}(\nabla u,\nabla u)\,dV\right|
\le \|\mathrm{Ric}\|_{L^\infty(\Omega)}\|\nabla u\|_{L^2(\Omega)}^2.
\]
Hence for $K$ sufficiently large the form $b_K$ is coercive on $S_2$. Subtract $K$ to recover the eigenvalues of the original problem
 
\subsubsection*{Variational characterizations}

The Euler-Lagrange conditions for the bilinear forms in \eqref{eq:correct_vib_bi_form} and \eqref{eq:correct_buc_bi_form} yield the eigensystems
\[
\left\{
\begin{aligned}
    & \lapl^2 u - \tau \lapl u = \Gamma_j u \quad \text{in } \Omega,\\[1mm]
    & u = 0 \hspace{2.7cm} \text{on } \partial \Omega,\\[1mm]
    & |\grad u| = 0 \hspace{2.12cm} \text{on } \partial \Omega,
\end{aligned}
\right.
\hspace{1cm}
\left\{
\begin{aligned}
    & \lapl^2 u + \rho u = \Lambda_j (-\lapl) u \quad \text{in } \Omega,\\[1mm]
    & u = 0 \hspace{3.34cm} \text{on } \partial \Omega,\\[1mm]
    & |\grad u| = 0 \hspace{2.76cm} \text{on } \partial \Omega,
\end{aligned}
\right.
\]
where the eigenvalues form discrete spectra
\[\Gamma_1 \leq \Gamma_2 \leq \Gamma_3 \leq \dots ,\]\[
\Lambda_1 \leq \Lambda_2 \leq \Lambda_3 \leq \dots\]
These spectra are respectively the clamped vibrating eigenvalues and the clamped buckling eigenvalues. The bilinear forms yield the variational characterizations for vibrating eigenvalues
\begin{equation}\label{eq:ultimate_final_vibrating_quotient}
\Gamma_j = \min_{\mathcal{L}} \max_{0 \neq u} \dfrac{\displaystyle \int_\Omega (\lapl u)^2 \, dV + \tau \int_\Omega |\grad u|^2 \, dV}{\displaystyle \int_\Omega |u|^2 \, dV},
\end{equation}
and for buckling eigenvalues
\begin{equation}\label{eq:ultimate_final_buckling_quotient}
\Lambda_j = \min_{\mathcal{L}} \max_{0 \neq u} \dfrac{\displaystyle \int_\Omega (\lapl u)^2 \, dV + \rho \int_\Omega u^2 \, dV}{\displaystyle \int_\Omega |\grad u|^2 \, dV},
\end{equation}
where $\mathcal{L}$ varies over $j$-dimensional subspaces of $H_0^2(\Omega)$ in both cases.
\section{Monotonicity proofs for Theorems \ref{th:vibratingmono}, \ref{th:bucklingneg}, \ref{th:bucklingpos}}

Our focus in this section is on the scaling properties of the vibrating and buckling eigenvalues for caps and geodesic disks. We will be considering the vibrating and buckling Rayleigh quotients, namely
\begin{equation}\label{eq:pre_vibratingmanifoldrq}
\dfrac{\displaystyle \int_{C(\Theta)} (\lapl u)^2 \, dV + \tau \int_{C(\Theta)} |\grad u|^2 \, dV} {\displaystyle \int_{C(\Theta)}  u^2 \, dV},
\end{equation}
\begin{equation}\label{eq:pre_bucklingmanifoldrq}
\dfrac{\displaystyle \int_{C(\Theta)} (\lapl u)^2 \, dV + \rho \int_{C(\Theta)} u^2 \, dV} {\displaystyle \int_{C(\Theta)} |\grad u|^2 \, dV}.
\end{equation}

We will map to the plane under polar coordinates via the conformal map $(\theta, \xi) \mapsto (\ta(\theta/2), \xi)$. The radial variable is $r = 
\ta (\theta/2).$ In the spherical case, this map is stereographic projection. Note that at the boundary, the normal derivative is the same as $\partial_\theta$. Define
\[
w_\pm(r) = \frac{4}{(1 \pm r^2)^2}.
\]Letting $v$ denote the pushforward of $u$ under such maps, we state the following transformation laws:
\begin{gather*}
    \lapl_{C(\Theta)} u = \frac{1}{w_\pm(r)}\lapl_{\R^2} v, \quad \partial_\theta u = \frac{1}{\sqrt{w_\pm(r)}}\partial_r v, \\dV = {w_\pm(r)} \,dA, \quad dS = \sqrt{w_\pm(r)}\, ds,
\end{gather*}
where $dV, dA$ are volume elements on the sphere and plane respectively, and $dS, ds$ the boundary elements on the sphere and plane respectively. One uses $w_+$ when pushing forward from the spherical cap, and $w_-$ from the hyperbolic disk. It follows that \eqref{eq:pre_vibratingmanifoldrq} transforms into
\[
\dfrac{\displaystyle \int_{\D(R)}\frac{1}{w_\pm(r)}(\lapl v)^2 \, dA + \tau \int_{\D(R)}|\grad v|^2 \, dA}
{\displaystyle \int_{\D(R)}  w_\pm(r)\, v^2 \, dA},
\]
where we have defined $R := \ta(\Theta/2).$ The buckling Rayleigh quotient \eqref{eq:pre_bucklingmanifoldrq} transforms similarly, i.e., the $(\lapl v)^2$ and $v^2$ terms acquire coefficients and the Dirichlet integral remains conformally invariant.

Now we linearly rescale the planar radius by writing $r = Rt$ (equivalently, $t=r/R$ so that $t \in [0,1]$ is the radial coordinate in the unit disk), which maps $\D(R)$ to the unit disk $\D$. After this rescaling, we continue to write $v$ for the resulting trial function on $\D$.

Recall that the auxiliary functions $\mu(\Theta)$ and $\nu(\Theta)$ introduced above encode the scaling and normalization factors: $\mu(\Theta)^4$ and $\mu(\Theta)^2$ scale the vibrating and buckling eigenvalues, while $\nu(\Theta)^2$ and $\nu(\Theta)^4$ normalize $\tau$ and $\rho$, respectively. Since $R=\ta(\Theta/2)$, we also view these as functions of $R$ by composition and write $\mu(R):=\mu(\Theta)$ and $\nu(R):=\nu(\Theta)$. (Strictly speaking, the actual factors are the powers $\mu(R)^4,\mu(R)^2,\nu(R)^2,\nu(R)^4$; we refer to $\mu(R)$ and $\nu(R)$ themselves as ``factors'' for brevity.)

Accordingly, the vibrating eigenvalues appearing in Theorem~\ref{th:vibratingmono} (i.e.\ the scaled/normalized quantities $\mu(R)^4\,\Gamma_j(\Theta,\tau/\nu(R)^2)$) have Rayleigh quotient
\begin{equation}\label{eq:vibratingscaledrq}
\dfrac{\displaystyle \int_{\D}\nu(R)^2\frac{1}{w_\pm(Rt)R^2}(\lapl v)^2 \,  dA +  \tau\int_{\D}|\grad v|^2 \, dA}{\displaystyle \int_{\D} \frac{\nu(R)^2}{\mu(R)^4}w_\pm(Rt)R^2v^2 \, dA},
\end{equation}
and the buckling eigenvalues appearing in Theorems~\ref{th:bucklingneg} and~\ref{th:bucklingpos} (i.e.\ the scaled/normalized quantities $\mu(R)^2\,\Lambda_j(\Theta,\rho/\nu(R)^4)$) have Rayleigh quotient
\begin{equation}\label{eq:bucklingscaledrq}
\dfrac{\displaystyle \int_{\D}\mu(R)^2\frac{1}{w_\pm(Rt)R^2}(\lapl v)^2 \,  dA +  \rho\int_{\D} \frac{\mu(R)^2}{\nu(R)^4}w_\pm(Rt)R^2v^2 \, dA}{\displaystyle \int_{\D} |\grad v|^2 \, dA}.
\end{equation}
The eigenvalue min--max principles in \eqref{eq:ultimate_final_vibrating_quotient} and \eqref{eq:ultimate_final_buckling_quotient} apply to these Rayleigh quotients on the trial space $H_0^2(\D)$, since the conformal map and the planar rescaling preserve the clamped boundary conditions.

We observe some appealing structure here. The monotonicity of the scaled eigenvalues in \eqref{eq:vibratingscaledrq} and \eqref{eq:bucklingscaledrq} depends only on the coefficient functions multiplying $(\lapl v)^2$ and $v^2$; if we interchange the roles of $\mu$ and $\nu$, then the coefficients for the vibrating eigenvalues become those for the buckling eigenvalues, and vice-versa.

Despite the various combinations of scaling and normalizing factors, there are only six coefficient functions whose monotonicity we need to analyze.  Since $R=\ta(\Theta/2)$, we have
\[
S(\Theta) := \si(\Theta) = \frac{2R}{1\pm R^2}, \qquad
T(\Theta) := 2\ta(\Theta/2) = 2R,
\]
where the $+$ sign corresponds to the spherical case and the $-$ sign corresponds to the hyperbolic case.

For fixed $t \in [0,1]$, set $u := R^2$ and define
\[
A_\pm(u) := 1 \pm u, \qquad B_\pm(u) := 1 \pm t^2 u.
\]
The admissible domain for spherical ($+$) is $u \in (0,\infty)$, and hyperbolic ($-$) is $u \in (0,1)$. In particular, on the admissible domain we have $A_\pm(u),B_\pm(u)>0$.
The six coefficient functions that we will need are
\begin{align*}
C_1(u) &:= B_\pm(u)^2, 
&\quad C_2(u) &:= B_\pm(u)^{-2},\\
C_3(u) &:= \left(\frac{B_\pm(u)}{A_\pm(u)}\right)^2, 
&\quad C_4(u) &:= \left(\frac{A_\pm(u)}{B_\pm(u)}\right)^2,\\
C_5(u) &:= \frac{A_\pm(u)^4}{B_\pm(u)^2},
&\quad C_6(u) &:= \frac{1}{A_\pm(u)^2B_\pm(u)^2}.
\end{align*}
Note that $C_2 = C_1^{-1}$ and $C_4 = C_3^{-1}$.

\begin{lemma}[Monotonicity of auxiliary functions]
For $t \in [0,1]$ and for $u$ in the admissible domain, the auxiliary functions have the following monotonicity:
\begin{itemize}
    \item $A_+(u)$ is increasing, $A_-(u)$ is decreasing;
    \item $B_+(u)$ is increasing, $B_-(u)$ is decreasing;
    \item the ratio $B_+(u)/A_+(u)$ is decreasing, and $B_-(u)/A_-(u)$ is increasing, with
    \[
    \frac{d}{du}\left(\frac{B_+(u)}{A_+(u)}\right) = -\frac{1-t^2}{A_+(u)^2},
    \qquad
    \frac{d}{du}\left(\frac{B_-(u)}{A_-(u)}\right) = +\frac{1-t^2}{A_-(u)^2};
    \]
    \item the ratio $A_+(u)/B_+(u)$ is increasing, and $A_-(u)/B_-(u)$ is decreasing, with
    \[
    \frac{d}{du}\left(\frac{A_+(u)}{B_+(u)}\right) = +\frac{1-t^2}{B_+(u)^2},
    \qquad
    \frac{d}{du}\left(\frac{A_-(u)}{B_-(u)}\right) = -\frac{1-t^2}{B_-(u)^2}.
    \]
\end{itemize}
\end{lemma}

The lemma is easily verified by direct differentiation. The monotonicity of $C_1(u),\ldots,C_6(u)$ follows immediately: powers preserve monotonicity, reciprocals reverse it (on the admissible domain where all functions are positive), and products of positive functions with the same monotonicity preserve it. Table~\ref{tab:coeff-monotonicity} summarizes the results.

\begin{table}[ht]
\centering
\begin{tabular}{@{}lcccccc@{}}
\toprule
Case & $C_1$ & $C_2$ & $C_3$ & $C_4$ & $C_5$ & $C_6$ \\
\midrule
$+$ (spherical) & $\uparrow$ & $\downarrow$ & $\downarrow$ & $\uparrow$ & $\uparrow$ & $\downarrow$ \\
$-$ (hyperbolic) & $\downarrow$ & $\uparrow$ & $\uparrow$ & $\downarrow$ & $\downarrow$ & $\uparrow$ \\
\bottomrule
\end{tabular}
\caption{Monotonicity summary for the six coefficients.}
\label{tab:coeff-monotonicity}
\end{table}

\subsubsection*{Monotonicity proofs for Theorems~\ref{th:vibratingmono}--\ref{th:bucklingpos}}

The variational characterizations for the scaled and normalized eigenvalues in Theorems~\ref{th:vibratingmono}--\ref{th:bucklingpos} read, respectively,
\begin{equation}\label{eq:rq_square}
\mu(\Theta)^4\,\Gamma_j\!\left(\Theta,\frac{\tau}{\nu(\Theta)^2}\right)
= \min_{\mathcal L}\ \max_{\substack{0\neq v\in\mathcal L}}
\frac{\displaystyle \int_{\D} V_1(R,t)\,(\Delta v)^2\,dA + \tau \int_{\D}|\nabla v|^2\,dA}
{\displaystyle \int_{\D} V_2(R,t)\,v^2\,dA},
\end{equation}
\begin{equation}\label{eq:rq_squaresquare}
\mu(\Theta)^2\,\Lambda_j\!\left(\Theta,\frac{\rho}{\nu(\Theta)^4}\right)
= \min_{\mathcal L}\ \max_{\substack{0\neq v\in\mathcal L}}
\frac{\displaystyle \int_{\D} B_1(R,t)\,(\Delta v)^2\,dA + \rho \int_{\D} B_2(R,t)\,v^2\,dA}
{\displaystyle \int_{\D}|\nabla v|^2\,dA},
\end{equation}
where $\mathcal L$ varies over $j$-dimensional subspaces of $H_0^2(\D)$, and the coefficient functions $V_1,V_2$ are determined by the vibrating Rayleigh quotient \eqref{eq:vibratingscaledrq} while $B_1,B_2$ are determined by the buckling Rayleigh quotient \eqref{eq:bucklingscaledrq}. In each case we use the substitutions $r=Rt$ and $u=R^2$.

Recall that in Theorems~\ref{th:vibratingmono}--\ref{th:bucklingneg} the scaling/normalizing factors $\mu(\Theta),\nu(\Theta)$ are chosen from
\[
S(\Theta)=\sin\Theta,\qquad T(\Theta)=2\tan(\Theta/2),
\]
and in Theorem~\ref{th:bucklingpos} we use mixed choices (e.g.\ in the first part of Theorem~\ref{th:bucklingpos} we take $\mu(\Theta)=S(\Theta)$ and $\nu(\Theta)=T(\Theta)$). 

We will phrase the monotonicity analysis in terms of $R^2$. In the spherical case ($+$), $\Theta\mapsto R=\tan(\Theta/2)$ is increasing on the relevant range, hence $u$ increases with $\Theta$. In the hyperbolic case ($-$), $\Theta$ increases from $-\infty$ to $0$ while $R$ decreases from $1$ to $0$, so $u$ decreases with $\Theta$; thus monotonicity in $u$ translates to the \emph{opposite} monotonicity in $\Theta$, in the hyperbolic setting.

Finally, for any nonzero $v\in H_0^2(\D)$ we have
\[
\int_{\D} v^2\,dA>0,\qquad \int_{\D}|\nabla v|^2\,dA>0,\qquad \int_{\D}(\Delta v)^2\,dA>0.
\]
The second and third inequalities hold because if $\Delta v=0$ or $|\grad v| = 0$ in $\D$, then
\[
\int_{\D} |\nabla v|^2\,dA=-\int_{\D} v\,\Delta v\,dA=0,
\]
hence $v$ is constant, and since its trace vanishes on $\partial\D$, we get $v\equiv 0$.

\medskip

\noindent\textbf{Case 1:} $\mu(\Theta)=S(\Theta)$, $\nu(\Theta)=S(\Theta)$.  
In this scenario (after rewriting the coefficients using $r=Rt$ and $u=R^2$), we obtain
\[
V_1=C_3,\qquad V_2=C_4,\qquad B_1=C_3,\qquad B_2=C_4,
\]
so the monotonicity of coefficients follows directly from Table~\ref{tab:coeff-monotonicity}.
In the spherical case, $C_3$ is strictly decreasing and $C_4$ is strictly increasing, so for every nonzero $v$ the quotient in \eqref{eq:rq_square} is strictly decreasing in $u$ when $\tau\ge 0$, and hence $\mu(\Theta)^4\Gamma_j(\Theta,\tau/\nu(\Theta)^2)$ is decreasing. Thus, the first part of Theorem~\ref{th:vibratingmono} is verified except for the strictness of the monotonicity which we address below.
For the buckling quotient \eqref{eq:rq_squaresquare}, if $\rho\le 0$ then at least one term in the numerator is strictly decreasing in $u$ (since $C_3$ is strictly decreasing and $\rho\,C_4$ is non-increasing), so the scaled buckling eigenvalue is decreasing as well. The first part of Theorem~\ref{th:bucklingneg} is verified aside from strictness of the monotonicity. 

With that, in the remaining cases, we will refrain from mentioning that strictness will be proved later.

\medskip

\noindent\textbf{Case 2:} $\mu(\Theta)=T(\Theta)$, $\nu(\Theta)=T(\Theta)$.  
Here we obtain
\[
V_1=C_1,\qquad V_2=C_2,\qquad B_1=C_1,\qquad B_2=C_2.
\]
Again by Table~\ref{tab:coeff-monotonicity}, in the spherical case $C_1$ is strictly increasing and $C_2$ is strictly decreasing, so the vibrating quotient \eqref{eq:rq_square} is strictly increasing in $u$ when $\tau\ge 0$, yielding monotonicity of the corresponding scaled vibrating eigenvalues (the second part of Theorem~\ref{th:vibratingmono}).  
For the buckling quotient \eqref{eq:rq_squaresquare}, if $\rho\le 0$ then at least one term in the numerator is strictly increasing in $u$ (since $C_1$ is strictly increasing and $\rho\,C_2$ is non-decreasing), hence the scaled buckling eigenvalues are increasing (the second part of Theorem~\ref{th:bucklingneg}).

\medskip

\noindent\textbf{Case 3:} $\mu(\Theta)=S(\Theta)$, $\nu(\Theta)=T(\Theta)$.  
For the \emph{buckling} coefficients we obtain
\[
B_1=C_3,\qquad B_2=C_6,
\]
so in the spherical case at least one term is strictly decreasing and in the hyperbolic case at least one term is strictly increasing; therefore, for $\rho\ge 0$ the numerator in \eqref{eq:rq_squaresquare} is strictly monotone in $u$, and the first part of Theorem~\ref{th:bucklingpos} follows.

\noindent \emph{Remark.} For the \emph{vibrating} coefficients in \eqref{eq:rq_square}, with $\mu(\Theta) = S(\Theta)$ and $\nu(\Theta) = T(\Theta)$, we find
\[
V_1=C_1,\qquad V_2=C_5.
\]
Since $V_1$ and $V_2$ have the same monotonicity, the coefficient table alone does not determine the monotonicity direction of the vibrating eigenvalues for this choice of $\mu$ and $\nu$; we therefore do not claim a vibrating monotonicity conclusion in this case.

\medskip

\noindent\textbf{Case 4:} $\mu(\Theta)=T(\Theta)$, $\nu(\Theta)=S(\Theta)$.  
For the \emph{buckling} coefficients we obtain
\[
B_1=C_1,\qquad B_2=C_5,
\]
so in the spherical case at least one term strictly increases and in the hyperbolic case at least one term strictly decreases; hence for $\rho\ge 0$ the numerator in \eqref{eq:rq_squaresquare} is strictly monotone in $u$, proving the second part of Theorem~\ref{th:bucklingpos}.

\noindent \emph{Remark.} For the \emph{vibrating} coefficients, with $\mu(\Theta) = T(\Theta)$ and $\nu(\Theta) = S(\Theta)$ we get
\[
V_1=C_3,\qquad V_2=C_6.
\]
As in Case~3, $V_1$ and $V_2$ here have the same monotonicity direction, so the coefficient table alone does not determine the monotonicity of the vibrating eigenvalues for this choice of $\mu$ and $\nu$.

\subsection*{Strict monotonicity of the eigenvalues}
Fix $0<R_1<R_2$, and let $Q_R(v)$ denote the relevant Rayleigh quotient on the fixed
trial space $H_0^2(\D)$. In each case considered above, the coefficient monotonicity shows that
\[
Q_{R_2}(v)<Q_{R_1}(v)
\]
for every nonzero trial function $v$ (or else the reverse inequality, depending on
the case).

Let $\mathcal L_1$ be a
$j$-dimensional subspace realizing the min--max value at $R_1$. Take $v$ to be the function that maximizes $Q_{R_2}$ over $\mathcal{L}_1$. By applying strict monotonicity to this $v$, we find
\[
\lambda_j(R_2)\le \max_{0\neq v\in\mathcal L_1} Q_{R_2}(v)
< \max_{0\neq v\in\mathcal L_1} Q_{R_1}(v)
= \lambda_j(R_1),
\]
and hence $\lambda_j(R_2)<\lambda_j(R_1)$ where $\lambda_j$ is the relevant eigenvalue. Argue similarly if the monotonicity goes in the opposite direction. Thus, strict monotonicity follows in all cases.
\newpage

\section{Limiting values}

As the geodesic disks expand to fill hyperbolic space $(\Theta \to -\infty)$ or spherical space $(\Theta \to \pi)$, the limiting values of the scaled and normalized eigenvalues depend on the geometry of the ambient space and our choice of scaling. In many cases, the spectrum collapses to 0 or explodes to $\infty$, but there are some non-standard "exotic" cases exhibiting convergence to finite non-zero values. Tables \ref{tab:vibrating-limits-conj}, \ref{tab:buckling-neg-limits-conj}, and \ref{tab:buckling-pos-mixed-limits-conj} provide an overview.


\begin{table}[ht]
  \centering
  \begin{tabular}{@{}lcc@{}}
    \toprule
    (scaling,\,normalize) & $\Theta\to-\infty$ & $\Theta\to\pi$ \\
    \midrule
    $(S,S)$ & $\Gamma \to \infty$ & $\Gamma \to 0$ \\
    $(T,T)$ & $\Gamma \to 1+\tau$ & $\Gamma \to \infty$ \\
    \bottomrule
  \end{tabular}
  \caption{Vibrating ($\tau\ge 0$): limiting behavior of the scaled eigenvalues in Theorem \ref{th:vibratingmono}.}
  \label{tab:vibrating-limits-conj}
\end{table}

\begin{table}[ht]
  \centering
  \begin{tabular}{@{}lcc@{}}
    \toprule
    (scaling,\,normalize) & $\Theta\to-\infty$ & $\Theta\to\pi$ \\
    \midrule
    $(S,S)$, $\rho<0$ & $\Lambda \to \infty$ & $\Lambda \to -\infty$ \\
    $(S,S)$, $\rho=0$ & $\Lambda \to \infty$ & $\Lambda \to 0$ \\
    $(T,T)$, $\rho\leq0$ & $\Lambda \to 1+\rho$ & $\Lambda \to \infty$ \\
    \bottomrule
  \end{tabular}
  \caption{Buckling ($\rho\le 0$): limiting behavior of the scaled eigenvalues in Theorem \ref{th:bucklingneg}.}
  \label{tab:buckling-neg-limits-conj}
\end{table}

\begin{table}[ht]
  \centering
  \begin{tabular}{@{}lcc@{}}
    \toprule
    (scaling,\,normalize) & $\Theta\to-\infty$ & $\Theta\to\pi$ \\
    \midrule
    $(S,T)$ & $\Lambda \to \infty$ & $\Lambda \to 0$ \\
    $(T,S)$ & $\Lambda \to 1$ & $\Lambda \to \infty$ \\
    \bottomrule
  \end{tabular}
  \caption{Buckling ($\rho\ge 0$): limiting behavior of the mixed-factor scaled eigenvalues in Theorem \ref{th:bucklingpos}.}
  \label{tab:buckling-pos-mixed-limits-conj}
\end{table}

Our proofs below go systematically row-by-row through the tables. We start with the vibrating case. To assist in these goals, we first prove two useful lemmas.

\begin{lemma}[Pointwise vanishing of the denominator]\label{lem:ptwisevanfourthorder}

Let $\Omega \subset \R^2$ be a bounded domain with smooth boundary. For each $R > 0$, let $w_R:\Omega \to (0,1]$ be a measurable function with $\text{ess inf } w_R > 0$. Fix $\tau \geq 0$ and $\rho \geq 0$. Let $\Gamma_k(w_R,\tau)$ and $\Lambda_k(w_R,\rho)$ be given by the variational characterizations
$$
\Gamma_k(w_R,\tau)
=
\min_{\mathcal L}\ \max_{\substack{0\neq v\in \mathcal L}}
\frac{\displaystyle \int_\Omega (\Delta v)^2\,dA + \tau \int_\Omega |\nabla v|^2\,dA}
{\displaystyle \int_\Omega v^2\,w_R\,dA},
$$
$$
\Lambda_k(w_R,\rho)
=
\min_{\mathcal L}\ \max_{\substack{0\neq v\in \mathcal L}}
\frac{\displaystyle \int_\Omega (\Delta v)^2\,dA + \rho \int_\Omega v^2\,dA}
{\displaystyle \int_\Omega |\nabla v|^2\,w_R\,dA},
$$
where $k \geq 1$ and $\mathcal L$ varies over $k$-dimensional subspaces of $H_0^2(\Omega)$. 

If $\lim_{R\to\infty} w_R = 0$ a.e., then $\lim_{R\to\infty} \Gamma_k(w_R,\tau)=\infty$ $\lim_{R\to\infty} \Lambda_k(w_R,\rho)=\infty$.
\end{lemma}

\begin{lemma}[Pointwise explosion of the numerator]\label{lem:ptwisevanfourthorder--numer}
Let $\Omega \subset \R^2$ be a bounded domain with smooth boundary. For each $R > 0$, let $w_R: \Omega\to (0,1]$ be a measurable function with $\text{ess inf } w_R > 0$. Fix $\tau \geq 0$ and $\rho \geq 0$. Let $\Gamma_k(w_R, \tau)$ and $\Lambda_k(w_R, \rho)$ be given by the variational characterizations
$$
\Gamma_k(w_R, \tau)
= \min_{\mathcal L}\ \max_{\substack{0\neq v\in\mathcal L}}
\frac{\displaystyle \int_{\Omega}\left((\Delta v)^2 + \tau |\nabla v|^2\right)w_R^{-1}\,dA}{\displaystyle \int_{\Omega} v^2\,dA},
$$
$$
\Lambda_k(w_R, \rho)
= \min_{\mathcal L}\ \max_{\substack{0\neq v\in\mathcal L}}
\frac{\displaystyle \int_{\Omega}\left((\Delta v)^2 + \rho v^2)\right)w_R^{-1}\,dA}{\displaystyle \int_{\Omega} |\nabla v|^2\,dA},
$$
where $k \geq 1$ and $\mathcal L$ varies over $k$-dimensional subspaces of $H_0^2(\Omega)$. 

If $\lim_{R\to\infty} w_R = 0$ a.e., then $\lim_{R\to\infty} \Gamma_k(w_R,\tau)=\infty$ and $\lim_{R\to\infty} \Lambda_k(w_R,\rho)=\infty$.
\end{lemma}
These lemmas are fourth-order generalizations of the weighted second-order argument used in \cite{H25}, which itself builds on \cite[Lemma 8]{LL22}. The motivation is that in our Rayleigh quotients \eqref{eq:rq_square} and \eqref{eq:rq_squaresquare}, some of the limiting cases will have a pointwise vanishing coefficient in the denominator or exploding coefficient in the numerator.

\begin{proof}[Proof of Lemma \ref{lem:ptwisevanfourthorder}]
The proof will be quite similar to the ones referenced above. Some details will be omitted if they are clear from these proofs. We let $0 < \delta < 1$ and define the weight $m_R := \max(\delta, w_R) \geq w_R$. The variational characterizations of the eigenvalues associated to the weight $m_R$ are
\begin{equation*}
\Gamma_k(m_R, \tau)
= \min_{\mathcal L}\ \max_{\substack{0\neq v\in\mathcal L}}
\frac{\displaystyle \int_{\Omega}(\Delta v)^2\,dA + \tau \int_{\Omega}|\nabla v|^2\,dA}{\displaystyle \int_{\Omega} v^2\,m_R\,dA},
\end{equation*}
\begin{equation*}
\Lambda_k(m_R, \rho)
= \min_{\mathcal L}\ \max_{\substack{0\neq v\in\mathcal L}}
\frac{\displaystyle \int_{\Omega}(\Delta v)^2\,dA + \rho \int_{\Omega}v^2\,dA}{\displaystyle \int_{\Omega} |\nabla v|^2\,m_R\,dA},
\end{equation*}
where $\mathcal{L}$ varies over $k$-dimensional subspaces of $H_0^2(\Omega)$. Since all terms in the numerators are nonnegative, we have
$$
\Gamma_k(w_R,\tau) \geq \Gamma_k(m_R,\tau),
$$
$$
\Lambda_k(w_R,\rho) \geq \Lambda_k(m_R,\rho).
$$
Since $\delta \le m_R \le 1$, inspection of the variational characterizations also shows that
\begin{equation}\label{lemmaweak1}
\delta^{-1}\Gamma_k(1, \tau) \geq \Gamma_k(m_R, \tau) \geq \Gamma_k(1, \tau),
\end{equation}
\begin{equation}\label{lemmaweak2}
\delta^{-1}\Lambda_k(1, \rho) \geq \Lambda_k(m_R, \rho) \geq \Lambda_k(1, \rho).
\end{equation}
Our goal is to show
\begin{equation}\label{lemmaweakgoal1}
\liminf_{R \to \infty} \Gamma_k(m_R,\tau)\geq \delta^{-1}\Gamma_1(1,\tau),
\end{equation}
\begin{equation}\label{lemmaweakgoal2}
\liminf_{R \to \infty} \Lambda_k(m_R,\rho)\geq \delta^{-1}\Lambda_1(1,\rho).
\end{equation}
Each quantity on the right-hand side is positive, and so the lemma follows from \eqref{lemmaweakgoal1} and \eqref{lemmaweakgoal2} by letting $\delta\to 0$.

Following along with \cite[Lemma 8]{LL22}, we generate weak eigenfunctions $f_R$ for $\Gamma_k(m_R, \tau)$ and $g_R$ for $\Lambda_k(m_R,\rho)$ that satisfy
$$
\int_\Omega \lapl \varphi \lapl f_R \, dA + \tau \int_\Omega \grad \varphi \cdot \grad f_R\, dA
= \Gamma_k(m_R, \tau)\int_\Omega \varphi f_R\, m_R\, dA,
$$
$$
\int_\Omega \lapl \varphi \lapl g_R \, dA + \rho \int_\Omega \varphi g_R\, dA
= \Lambda_k(m_R, \rho)\int_\Omega \grad \varphi\cdot \grad g_R\, m_R\, dA.
$$
We normalize by $\int_\Omega f_R^2\,dA = 1$ and $\int_\Omega |\grad g_R|^2\,dA = 1$ for each $R$. Then the eigenfunctions are uniformly $H^2$-bounded in $R$ for both cases, by choosing $\varphi=f_R$ and $\varphi=g_R$ and using \eqref{lemmaweak1} and \eqref{lemmaweak2}. By the Rellich-Kondrachov theorem, after passing to a subsequence as $R\to\infty$, there exist functions $f,g \in H_0^2(\Omega)$ such that
$f_R \to f$ weakly in $H^2$ and strongly in $L^2$, and $g_R \to g$ weakly in $H^2$ and strongly in $H^1$. The key point for proving \eqref{lemmaweakgoal1} and \eqref{lemmaweakgoal2} is that $m_R \to \delta$ a.e. as $R\to\infty$, since $w_R \to 0$ a.e. The rest of the proof of \cite[Lemma 8]{LL22} now applies, with the slight modification that the argument for $g_R$ uses strong $H^1$ convergence. This proves \eqref{lemmaweakgoal1} and \eqref{lemmaweakgoal2}.
\end{proof}
\begin{proof}[Proof of Lemma \ref{lem:ptwisevanfourthorder--numer}] This time, the variational characterizations of the eigenvalues associated to the weight $m_R$ are
$$
\Gamma_k(m_R, \tau)
= \min_{\mathcal L}\ \max_{\substack{0\neq v\in\mathcal L}}
\frac{\displaystyle \int_{\Omega}\left((\Delta v)^2 + \tau |\nabla v|^2\right)m_R^{-1}\,dA}{\displaystyle \int_{\Omega} v^2\,dA},
$$
$$
\Lambda_k(m_R, \rho)
= \min_{\mathcal L}\ \max_{\substack{0\neq v\in\mathcal L}}
\frac{\displaystyle \int_{\Omega}\left((\Delta v)^2 + \rho v^2\right)m_R^{-1}\,dA}{\displaystyle \int_{\Omega} |\nabla v|^2\,dA},
$$
where $\mathcal{L}$ varies over $k$-dimensional subspaces of $H_0^2(\Omega)$ and all terms are nonnegative since $\tau, \rho \geq 0$. Clearly, $\Gamma_k(w_R,\tau) \geq \Gamma_k(m_R,\tau)$ and $\Lambda_k(w_R,\rho) \geq \Lambda_k(m_R,\rho)$.  Inspection of the variational characterizations also shows that
\begin{equation}\label{lemmaweak1num}
\delta^{-1}\Gamma_k(1, \tau) \geq \Gamma_k(m_R, \tau) \geq \Gamma_k(1, \tau),
\end{equation}
\begin{equation}\label{lemmaweak2num}
\delta^{-1}\Lambda_k(1, \rho) \geq \Lambda_k(m_R, \rho) \geq \Lambda_k(1, \rho).
\end{equation}
Our goal is to show
\begin{equation}\label{eq:weakgoalexpl1}
\liminf_{R \to \infty} \Gamma_k(m_R,\tau)\geq \delta^{-1}\Gamma_1(1,\tau),
\end{equation}
\begin{equation}\label{eq:weakgoalexpl2}
\liminf_{R \to \infty} \Lambda_k(m_R,\rho)\geq \delta^{-1}\Lambda_1(1,\rho),
\end{equation}
after which the lemma follows by letting $\delta \to 0$.

Following along with the proof of Lemma \ref{lem:ptwisevanfourthorder}, we generate weak eigenfunctions $f_R$ for $\Gamma_k(m_R, \tau)$ and $g_R$ for $\Lambda_k(m_R,\rho)$ that satisfy
$$
\int_\Omega \frac{\lapl \varphi \lapl f_R + \tau \grad \varphi \cdot \grad f_R}{m_R}\, dA
= \Gamma_k(m_R, \tau) \int_\Omega \varphi f_R \, dA,
$$
$$
\int_\Omega \frac{\lapl \varphi \lapl g_R + \rho \varphi g_R}{m_R}\, dA
= \Lambda_k(m_R, \rho) \int_\Omega \grad \varphi \cdot \grad g_R \, dA.
$$
We normalize by $\int_\Omega f_R^2 = 1$ and $\int_\Omega |\grad g_R|^2 = 1$ for each $R$, and in both cases, we generate families of functions uniformly bounded in $H_0^2(\Omega)$ by choosing $\varphi = f_R$ and $\varphi = g_R$ and using \eqref{lemmaweak1num} and \eqref{lemmaweak2num}. The Rellich-Kondrachov theorem yields two functions $f,g \in H_0^2(\Omega)$ so that $f_R \to f$ weakly in $H^2$ and strongly in $L^2$, and $g_R \to g$ weakly in $H^2$ and strongly in $H^1$.

The rest of the proof is identical to that of Lemma \ref{lem:ptwisevanfourthorder}, with the single modification that we use $m_R^{-1}\to \delta^{-1}$ a.e. in place of $m_R \to \delta$ a.e., as $R \to \infty$, hence arriving at \eqref{eq:weakgoalexpl1} and \eqref{eq:weakgoalexpl2}. Letting $\delta \to 0$, we conclude that the eigenvalues of the weighted operators tend to $\infty$ as $R \to \infty$.
\end{proof}

\subsection*{Vibrating limits (Theorem \ref{th:vibratingmono} and Table \ref{tab:vibrating-limits-conj} for $\tau \geq 0$)}

\subsubsection*{Hyperbolic $(S,S)$: $\Theta \to -\infty$}

First consider the vibrating problem with scaling and normalizing factor $S(\Theta)$. The variational characterization \eqref{eq:rq_square} in this case reads 

\[
\Gamma_j(\Theta, \tau/\sinh^2(\Theta))\sinh^4(\Theta) =\min_{\mathcal{L}}\max_{0 \neq v \in \mathcal{L}}\dfrac{\displaystyle \int_{\D}\frac{(1-R^2t^2)^2}{(1-R^2)^2}(\lapl v)^2 \,  dA +  \tau\int_{\D}|\grad v|^2 \, dA}{\displaystyle \int_{\D} \frac{(1-R^2)^2}{(1-R^2t^2)^2}v^2 \, dA}
\]
where $\mathcal{L}$ varies over $j$-dimensional subspaces of $H_0^2(\D)$ and $R = \tanh(|\Theta|/2).$ (For brevity, we will from now on only display the Rayleigh quotient for our limits with the understanding that the $j$-th scaled/normalized eigenvalue arises from the min-max characterization of said Rayleigh quotient.) Our goal is to show this quantity goes to $\infty$ as $R \to 1$.

We may bound the first integral in the numerator below by $\int_\D (\lapl v)^2\, dA$. Then, we see the weight $(1-R^2)^2/(1-R^2t^2)^2$ satisfies the hypotheses of Lemma \ref{lem:ptwisevanfourthorder} (just with $R \to 1$ instead of $R \to \infty$). Applying the lemma, we conclude that all eigenvalues tend to $\infty$.

\subsubsection*{Spherical $(S,S)$: $\Theta\to\pi$}

Write
\[
\mathcal{Q}_R(v):=
\frac{\displaystyle \int_{\D}\frac{(1+R^2t^2)^2}{(1+R^2)^2}(\Delta v)^2\,dA
+\tau\int_{\D}|\nabla v|^2\,dA}
{\displaystyle \int_{\D}\frac{(1+R^2)^2}{(1+R^2t^2)^2}v^2\,dA}
\]
with \(\tau \geq 0\) and \(0 \neq v \in H_0^2(\D)\).

Fix \(k\) and a nonzero bump function \(\eta\in C_c^\infty(B(0,1))\).
Observe for every \(R>0\) that the ball
$\D(R^{-1})=\{t \leq R^{-1}\}$ contains \(k\) pairwise disjoint balls of radius
\(1/kR\). Let \(x_1,\dots,x_k\) be their centers, and define
\[
\eta_i(x):=\eta\!\left(kR(x-x_i)\right), \qquad i=1,\dots,k.
\]
Then \(\operatorname{supp}\eta_i\subset B(x_i,1/kR)\), so the supports are
pairwise disjoint and contained in \(\{t\le R^{-1}\}\subset\D\).
Set \(V_k:=\mathrm{span}\{\eta_1,\dots,\eta_k\}\subset H_0^2(\D)\).

The $\eta_i$ are $H^2$-orthogonal and so by standard scaling laws,
\[
\int_\D \eta_i^2\,dA \sim R^{-2},\qquad
\int_\D |\nabla \eta_i|^2\,dA \sim 1,\qquad
\int_\D (\Delta \eta_i)^2\,dA \sim R^{2},
\]
where the implicit constants depend only on \(\eta\) and \(k\).
For \(t \le R^{-1}\) we have
\[
\frac{(1+R^2t^2)^2}{(1+R^2)^2}\le \frac{4}{(1+R^2)^2}\lesssim R^{-4},
\qquad
\frac{(1+R^2)^2}{(1+R^2t^2)^2}\ge \frac{(1+R^2)^2}{4}\gtrsim R^{4}.
\]
Therefore, for \(v=\sum c_i\eta_i\),
\[
\int_{\D}\frac{(1+R^2)^2}{(1+R^2t^2)^2}v^2\,dA
\;\gtrsim\;
R^4\int_\D v^2\,dA
\;\sim\;
R^2\sum_{i=1}^k|c_i|^2,
\]
\[
\int_{\D}\frac{(1+R^2t^2)^2}{(1+R^2)^2}(\Delta v)^2\,dA
\;\lesssim\;
R^{-4}\int_\D (\Delta v)^2\,dA
\;\sim\;
R^{-2}\sum_{i=1}^k|c_i|^2,
\]
and
\[
\int_\D |\nabla v|^2\,dA \sim \sum_{i=1}^k|c_i|^2.
\]
Hence, as $R \to \infty$,
\[
0 \leq \max_{0\neq v\in V_{k}} \mathcal{Q}_R(v)
\;\le\;
\frac{C_1R^{-2}+C_2\tau}{C_3R^2}
\to 0,
\]
where \(C_1,C_2,C_3\) are independent of \(R\).
By the min--max characterization, the $k$-th eigenvalue must tend to \(0\) as \(\Theta\to\pi\).
\subsubsection*{Hyperbolic $(T,T)$: $\Theta \to -\infty$}

Now we look at scaling and normalizing factor $T(\Theta)$.

Before applying transformations, the scaled and normalized eigenvalue has characterization
\begin{equation}\label{eq:vibratingmanifoldrq}
\Gamma_j(\Theta, \tau/T(\Theta)^2)T(\Theta)^4 = \min_{\mathcal{L}}\max_{0 \neq u \in \mathcal{L}}\dfrac{T(\Theta)^4\displaystyle \int_{C(\Theta)} (\lapl u)^2 \, dV +  T(\Theta)^2 \tau\int_{C(\Theta)} |\grad u|^2 \, dV} {\displaystyle \int_{C(\Theta)}  u^2 \, dV},
\end{equation}
where $T(\Theta) = 2\tanh(|\Theta|/2)$ and $\mathcal{L}$ ranges over $j$-dimensional subspaces of $H_0^2(C(\Theta))$.
Sending $\Theta \to -\infty$ shows that this expression formally appears to approach a biharmonic operator on hyperbolic space with a tensile parameter. Since $T(\Theta)\to 2$, the limiting value is affected by this constant factor. Our claim is that each of these scaled eigenvalues tends towards $1 + \tau$.
We first show this value is a lower bound for the limit, and then we show it is an upper bound.

First, observe that a function $u \in H_0^2(C(\Theta))$ can be extended to all of hyperbolic space $\H$ by defining it as $0$ outside $C(\Theta)$. Call such an extension $\tilde{u}$. By the calculations in \cite[p.\,1]{M70}, we see
\begin{equation}\label{eq:hypspeclowbound--grad}
\frac{1}{4}\int_{C(\Theta)} u^2 \, dV=\frac{1}{4}\int_\H \tilde{u}^2 \, dV \leq \int_\H |\grad \tilde{u}|^2 \, dV = \int_{C(\Theta)} |\grad u|^2 \, dV.
\end{equation}
Using Cauchy--Schwarz and integrating by parts, we calculate
\[
\left(\int_{C(\Theta)} |\grad u|^2 \, dV\right)^2= \left( \int_{C(\Theta)} u(-\lapl u) \,dV \right)^2 \leq \int_{C(\Theta)} (\lapl u)^2 \,dV \, \int_{C(\Theta)}u^2 \,dV.
\]
Applying Equation \eqref{eq:hypspeclowbound--grad} to the left-hand side and dividing by $\int_{C(\Theta)}u^2 \,dV$ yields the lower bound
\begin{equation}\label{eq:hypspeclowbound--lapl}
\frac{1}{16}\int_{C(\Theta)} u^2 \, dV\leq \int_{C(\Theta)} (\lapl u)^2 \, dV.
\end{equation}
Thus, the Rayleigh quotient in \eqref{eq:vibratingmanifoldrq} has the universal lower bound
\[
\frac{1}{16}T(\Theta)^4 + \frac{\tau}{4}T(\Theta)^2.
\]
Letting $\Theta \to -\infty$, the universal lower bound converges to $1+ \tau$ and does not depend on the variational subspace chosen. We conclude for every $j \geq 1$ that
\[
\liminf_{\Theta \to -\infty} \Gamma_j(\Theta, \tau/T(\Theta)^2)T(\Theta)^4 \geq 1+ \tau.
\]

Now, we show the upper bound. Consider the Laplace-Beltrami operator $-\lapl_\H$ on all of hyperbolic space. Let $P_{-1/2}$ represent the Legendre function of degree $-1/2$. Recall our coordinate system $z \simeq (\theta,\xi)$ where $\theta$ is geodesic radius. We consider the radial function $\psi(\theta) = P_{-1/2}(\cosh \theta)$ which satisfies the pointwise eigenfunction equation
\[
-\lapl_\H \psi = \frac{1}{4}\psi.
\]
Note that $\psi(\theta) \sim \theta e^{-\theta/2}$ for large $\theta$ and is positive everywhere \cite[Section 14.8, Eq 14.8.14]{DLMF}. Further, using \cite[Section 14.10]{DLMF} and the asymptotic $P_{1/2}(x) = O(\sqrt x)$, it is not too hard to derive that 

\begin{equation}\label{eq:legendredevbound}
|\grad_\H \psi(\theta)| = |\partial_\theta (P_{-1/2}(\cosh \theta))| \lesssim \theta e^{-\theta/2} \sim \psi(\theta). 
\end{equation}

We will use the form of the Rayleigh quotient in \eqref{eq:vibratingmanifoldrq} for the proof. Let $C(\Theta)$ be the geodesic disk and set $\Theta'=\Theta+1$ (recall that $\Theta<0$ in the hyperbolic case), so that $C(\Theta')$ is the concentric inner disk of radius $|\Theta|-1$. Fix the transition annulus of width 1,
$A(\Theta) := C(\Theta)\setminus C(\Theta')$. Let $\eta$ be a smooth, compactly supported radial cutoff function on $C(\Theta)$ that equals $1$ on $C(\Theta')$ and decays to 0 across $A(\Theta)$.  We supply $f = \eta \, \psi$ as a trial function into the variational characterization for the first eigenvalue. Since $f$ agrees with the pointwise eigenfunction $\psi$ on $C(\Theta+1)$, the only error is generated on the transition annulus, and we will show the error is well-controlled despite the exponential volume growth.

We start by controlling the gradient term in the Rayleigh quotient. First observe that by Green's identity applied to $\eta^2\psi$ and $\psi$, we obtain
\begin{align*}
-\int_{C(\Theta)} \eta^2\psi\,\Delta\psi\,dV
&= \int_{C(\Theta)} \langle \grad(\eta^2\psi),\grad\psi\rangle_g\,dV \\
&= \int_{C(\Theta)} \eta^2|\grad\psi|^2\,dV
  + \int_{C(\Theta)} \psi\langle \grad(\eta^2),\grad\psi\rangle_g\,dV \\
&= \int_{C(\Theta)} \eta^2|\grad\psi|^2\,dV
  + 2\int_{C(\Theta)} \eta\psi\langle \grad\eta,\grad\psi\rangle_g\,dV.
\end{align*}
Then, applying the displayed equation and expanding $|\grad f|^2$ gives
\begin{align*}
\int_{C(\Theta)} |\grad f|^2\,dV
&= \int_{C(\Theta)} \eta^2|\grad\psi|^2\,dV + 2\int_{C(\Theta)} \eta\psi\langle \grad\eta,\grad\psi\rangle_g\,dV
 + \int_{C(\Theta)} \psi^2|\grad\eta|^2\,dV
  \notag\\
&=
-\int_{C(\Theta)} \eta^2\psi\,\Delta\psi\,dV
   + \int_{C(\Theta)} \psi^2|\grad\eta|^2\,dV \notag\\
&= \frac14\int_{C(\Theta)} f^2\,dV
   + \int_{C(\Theta)} \psi^2|\grad\eta|^2\,dV
\end{align*}
where we have used the fact that $-\lapl \psi = \frac{1}{4}\psi$.  As $\grad \eta$ is only supported on the transition annulus, the gradient term has the bound
\[
\int_{C(\Theta)} \psi^2|\grad \eta|^2 \,dV \lesssim \int_{A(\Theta)} \psi^2 \, dV
\]
and we deduce that
\begin{equation}\label{eq:vibTTgrad-id}
0 \leq \int_{C(\Theta)} |\grad f|^2 \, dV - \frac{1}{4} \int_{C(\Theta)} f^2 \, dV \lesssim \int_{A(\Theta)} \psi^2 \,dV.
\end{equation}
For the Laplacian term, observe that
\[
-\lapl f-\frac14 f
= -\lapl(\eta\psi)-\frac14\eta\psi
= -\psi\,\lapl\eta-2\langle \grad\eta,\grad\psi\rangle_g,
\]
since $-\lapl\psi=\frac14\psi$. Using that $-\lapl f-\frac14 f$ is only supported on
$A(\Theta)$ and $\grad\eta$ and $\lapl\eta$ are uniformly bounded there, we find
\begin{equation}\label{eq:vibTTlapl-error}
\int_{C(\Theta)}\left(-\lapl f-\frac14 f\right)^2\,dV
\lesssim
\int_{A(\Theta)} \bigl(|\grad\psi|^2+\psi^2\bigr)\,dV \lesssim \int_{A(\Theta)} \psi^2 \,dV
\end{equation}
after applying Equation \eqref{eq:legendredevbound}.
Now, by Green's identity,
\begin{align*}
\int_{C(\Theta)}\left(-\lapl f-\frac14 f\right)^2\,dV
&=
\int_{C(\Theta)} (\lapl f)^2\,dV
+\frac12\int_{C(\Theta)} f\,\lapl f\,dV
+\frac1{16}\int_{C(\Theta)} f^2\,dV \\
&=
\int_{C(\Theta)} (\lapl f)^2\,dV
-\frac12\int_{C(\Theta)} |\grad f|^2\,dV
+\frac1{16}\int_{C(\Theta)} f^2\,dV.
\end{align*}
Rearranging, we obtain
\[
\int_{C(\Theta)} (\lapl f)^2\,dV-\frac1{16}\int_{C(\Theta)} f^2\,dV
=
\int_{C(\Theta)}\left(-\lapl f-\frac14 f\right)^2\,dV
+\frac12\left(
\int_{C(\Theta)} |\grad f|^2\,dV-\frac14\int_{C(\Theta)} f^2\,dV
\right).
\]
Using \eqref{eq:vibTTgrad-id} and \eqref{eq:vibTTlapl-error}, we deduce
\[
0\le
\int_{C(\Theta)} (\lapl f)^2\,dV-\frac1{16}\int_{C(\Theta)} f^2\,dV
\lesssim
\int_{A(\Theta)} \psi^2\,dV.
\]
The asymptotics imply $\psi(\theta)^2\sinh\theta \asymp \theta^2$. (Here, $\asymp$ means $\lesssim$ and $\gtrsim$.) Therefore,
\[
\int_{A(\Theta)}\psi^2\,dV
=2\pi\int_{|\Theta|-1}^{|\Theta|}\psi(\theta)^2\sinh\theta\,d\theta
\lesssim \int_{|\Theta|-1}^{|\Theta|}\theta^2\,d\theta
\lesssim |\Theta|^2,
\]
\[
\int_{C(\Theta)}f^2\,dV
\ge \int_{C(\Theta+1)}\psi^2\,dV
=2\pi\int_{0}^{|\Theta|-1}\psi(\theta)^2\sinh\theta\,d\theta
\gtrsim \int_{1}^{|\Theta|-1}\theta^2\,d\theta
\gtrsim |\Theta|^3
\]
using that $f=\psi$ on $C(\Theta+1)$.

As a result,
\[
\frac{\int_{A(\Theta)} \psi^2\,dV}{\int_{C(\Theta)} f^2\,dV}\to 0
\qquad\text{as }\Theta\to-\infty.
\]
Hence, dividing \eqref{eq:vibTTgrad-id} by $\int_{C(\Theta)}f^2\,dV$ yields
\begin{equation}\label{eq:vibTTgradratio}
\limsup_{\Theta \to -\infty}\frac{\int_{C(\Theta)} |\grad f|^2 \, dV}{\int_{C(\Theta)} f^2 \,dV} \leq \frac14.
\end{equation}
Likewise,
\begin{equation}\label{eq:vibTTlaplratio}
\limsup_{\Theta \to -\infty}\frac{\int_{C(\Theta)} (\lapl f)^2 \, dV}{\int_{C(\Theta)} f^2 \,dV} \leq \frac{1}{16}.
\end{equation}
Moreover, by \eqref{eq:hypspeclowbound--grad} and \eqref{eq:hypspeclowbound--lapl}, the corresponding lower bounds match these upper bounds, so the $\limsup$ in \eqref{eq:vibTTgradratio} and \eqref{eq:vibTTlaplratio} are in fact limits.

Using \eqref{eq:vibTTgradratio} and \eqref{eq:vibTTlaplratio} in the Rayleigh quotient in \eqref{eq:vibratingmanifoldrq} yields
\[
\limsup_{\Theta \to -\infty} \frac{T(\Theta)^4\int_{C(\Theta)} (\lapl f)^2 \, dV + T(\Theta)^2\tau\int_{C(\Theta)} |\grad f|^2 \, dV}{\int_{C(\Theta)} f^2 \,dV} \leq 16\cdot\frac{1}{16}+4\tau\cdot\frac14=1+\tau.
\]
It follows that for this specifically chosen trial function $f$, the Rayleigh quotient tends to $1+\tau$ as $\Theta \to -\infty$. Thus, for the lowest eigenvalue, we have that
\[
\limsup_{\Theta \to -\infty} \Gamma_1(\Theta, \tau/T(\Theta)^2)T(\Theta)^4 \leq 1+\tau
\]
and hence its limit is $1 + \tau$, since we already showed $1 + \tau$ is a lower bound.

To conclude the above upper bound for the $j$-th eigenvalue, we can enlarge the geodesic disk enough to contain $j$ copies of the trial function $f$ with disjoint supports. As hyperbolic space has infinite volume, we can always accomplish such an enlargement by taking $\Theta$ sufficiently negative. These disjoint copies will thus be $H^2$-orthogonal and so the variational characterization implies that the $j$-th eigenvalue has upper bound of $1+\tau$ in the limit as $\Theta \to -\infty$. As showed previously, this value is a lower bound and the claim follows.

\subsubsection*{Spherical $(T,T)$: $\Theta \to \pi$}

The Rayleigh quotient here reads
\[
\dfrac{\displaystyle \int_{\D}(1 + R^2t^2)^2(\lapl v)^2 \,  dA +  \tau\int_{\D}|\grad v|^2 \, dA}{\displaystyle \int_{\D}\frac{1}{(1 + R^2t^2)^2}v^2 \, dA}
\]
The Laplacian coefficient is bounded below by $1$ and the $v^2$ coefficient satisfies the hypotheses of Lemma \ref{lem:ptwisevanfourthorder}. By the lemma, all eigenvalues tend to $\infty$ as $R \to \infty$.

\subsection*{Buckling limits (Theorem \ref{th:bucklingneg} and Table \ref{tab:buckling-neg-limits-conj} for $\rho \leq 0$)}

\subsubsection*{Hyperbolic $(S,S)$: $\Theta \to -\infty$}

We consider the buckling problem with scaling and normalizing factor $S(\Theta)$ and $\rho \leq 0$. The first part of Theorem \ref{th:bucklingneg} asserts that this eigenvalue is decreasing.

The Rayleigh quotient reads
\[
\dfrac{\displaystyle \int_{\D}\frac{(1- R^2t^2)^2}{(1 -  R^2)^2}(\lapl v)^2 \,  dA +\rho\int_{\D}\frac{(1 -  R^2)^2}{(1 -  R^2t^2)^2}v^2 \, dA}{\displaystyle \int_{\D} |\grad v|^2 dA}.
\]
Write $\varepsilon = 1-R^2$ and $y = 1-t^2$. Note that $y \sim \text{dist}(\partial \D, z)$ with $z = te^{i\theta}$. Then we have the bound
\[
\frac{(1-R^2t^2)^2}{(1-R^2)^2} \geq y^2/\varepsilon^2. 
\]

Recall the Hardy-type bound $\int |\grad v|^2 \geq C_1 \int v^2/y^2$ for some constant $C_1 > 0$. Since $v$ is clamped, we can integrate by parts to see
\[
\int |\grad v|^2 \leq \int |v||\lapl v| \leq \left (\int (v^2/y^2)\int y^2(\lapl v)^2 \right)^{1/2} \leq \left (C_1^{-1}\int |\grad v|^2\int y^2(\lapl v)^2 \right)^{1/2}.
\]
Then, we get a lower bound on the first term of the numerator of
\[
\geq C_1 \varepsilon^{-2} \int |\grad v|^2.
\]

For the second term, observe that $1-R^2 \leq 1-R^2t^2$. Using that $\rho \leq 0$, we derive a lower bound 
\[
\rho\int_\D \frac{(1-R^2)^2}{(1-R^2t^2)^2}v^2 \, dA \geq \rho\int_\D v^2 \, dA \geq C_2 \,\rho \int |\grad v|^2 \, dA
\]
for some $C_2 > 0$ by applying the Poincaré inequality.

Therefore, the Rayleigh quotient is bounded below by
\[
C_1 \varepsilon^{-2} + C_2\rho
\]
for universal positive constants $C_1, C_2$. Sending $R \to 1$ ($\varepsilon \to 0$) shows the minimum of the Rayleigh quotient tends to $\infty$. Thus the first eigenvalue tends to $\infty$ by the variational characterization, and hence so do all higher eigenvalues and the claim follows. 

\subsubsection*{Spherical $(S,S)$: $\Theta \to \pi$}

The Rayleigh quotient here is
\[
\mathcal Q_R(v):=
\dfrac{\displaystyle \int_{\D}\frac{(1+ R^2t^2)^2}{(1+ R^2)^2}(\lapl v)^2 \,  dA +\rho\int_{\D}\frac{(1+  R^2)^2}{(1+  R^2t^2)^2}v^2 \, dA}{\displaystyle \int_{\D} |\grad v|^2 \, dA}.
\]

Fix $k\geq1$ and use the same $k$-dimensional trial space as in the spherical $(S,S)$ vibrating limit. Thus
\[
V_k=\operatorname{span}\{\eta_1,\dots,\eta_k\}\subset H_0^2(\D),
\]
where the $\eta_i$ have pairwise disjoint supports contained in $\{t\leq R^{-1}\}$. For $v=\sum_{i=1}^k c_i\eta_i$, the scaling estimates from that proof give
\[
\int_\D v^2\,dA\asymp R^{-2}\sum_{i=1}^k|c_i|^2,\qquad
\int_\D|\grad v|^2\,dA\asymp\sum_{i=1}^k|c_i|^2,
\qquad
\int_\D(\lapl v)^2\,dA\asymp R^2\sum_{i=1}^k|c_i|^2.
\]
On these supports,
\[
\frac{(1+R^2t^2)^2}{(1+R^2)^2}\lesssim R^{-4},
\qquad
\frac{(1+R^2)^2}{(1+R^2t^2)^2}\gtrsim R^4.
\]
If $\rho<0$, multiplication of the second inequality by $\rho$ reverses its direction, and hence
\[
\max_{0\neq v\in V_k}\mathcal Q_R(v)
\leq C_1R^{-2}-C_2|\rho|R^2\longrightarrow-\infty.
\]
The min--max characterization therefore gives
\[
\Lambda_k(\Theta,\rho/S(\Theta)^4)S(\Theta)^2\longrightarrow-\infty
\qquad\text{as }\Theta\to\pi.
\]
When $\rho=0$, the quotient is nonnegative, while the same trial space gives
\[
0\leq \Lambda_k(\Theta,0)S(\Theta)^2
\leq \max_{0\neq v\in V_k}\mathcal Q_R(v)
\lesssim R^{-2}\longrightarrow0.
\]

\begin{remark}[Exterior interpretation]\label{rem:exterior-buckling}
Formally letting $R\to\infty$ in the preceding quotient and applying the inversion $z\mapsto1/z$ produces the exterior buckling quotient
\[
\dfrac{\displaystyle \int_{\R^2\setminus\D}(\lapl v)^2\,dA
+\rho\int_{\R^2\setminus\D}v^2\,dA}
{\displaystyle \int_{\R^2\setminus\D}|\grad v|^2\,dA}.
\]
The quotient above represents the buckling analogue of the exterior Robin limit in \cite{H25}, but with a different outcome. When $\rho<0$, the exterior quotient is unbounded below and does not yield finite discrete negative eigenvalues. As a curiosity (or perhaps it is related), this limiting case is the only one in the paper in which every fixed-index eigenvalue tends to $-\infty$. Related exterior problems are studied in \cite{AF04,AF05,MM21,MM23}.
\end{remark}

\subsubsection*{Hyperbolic $(T,T)$: $\Theta \to -\infty$} 

Now we consider scaling and normalizing factor $T(\Theta)$, which is increasing by Theorem \ref{th:bucklingneg}. The proof here is similar to the one for the hyperbolic $(T,T)$ vibrating limit. We claim for each fixed $j\ge 1$ that
\[
\lim_{\Theta\to-\infty}\Lambda_j(\Theta,\rho/T(\Theta)^4)\,T(\Theta)^2=1+\rho.
\]

Before applying transformations, the scaled and normalized eigenvalue has characterization
\begin{equation*}
\Lambda_j(\Theta, \rho/T(\Theta)^4)T(\Theta)^2 = \min_{\mathcal{L}}\max_{0 \neq u \in \mathcal{L}}\dfrac{T(\Theta)^2\displaystyle \int_{C(\Theta)} (\lapl u)^2 \, dV +  T(\Theta)^{-2} \rho\int_{C(\Theta)} | u|^2 \, dV} {\displaystyle \int_{C(\Theta)}  |\grad u|^2 \, dV}.
\end{equation*}

\noindent Fix $u\in H_0^2(C(\Theta))$. Using Equations \eqref{eq:hypspeclowbound--grad} and \eqref{eq:hypspeclowbound--lapl} and the fact that $\rho \leq 0$, we deduce
\[
\dfrac{T(\Theta)^2\displaystyle \int_{C(\Theta)} (\lapl u)^2 \, dV +  T(\Theta)^{-2} \rho\int_{C(\Theta)} u^2 \, dV} {\displaystyle \int_{C(\Theta)}  |\grad u|^2 \, dV}
\ge \frac{T(\Theta)^2}{4}+\frac{4\rho}{T(\Theta)^2}.
\]
Hence $\Lambda_j(\Theta,\rho/T(\Theta)^4)\,T(\Theta)^2 \ge \frac{T(\Theta)^2}{4}+\frac{4\rho}{T(\Theta)^2}$ and letting $\Theta\to-\infty$ yields the universal lower bound of $1 + \rho$ for the Rayleigh quotient. Thus for every $j \geq 1$,
\[
\liminf _{\Theta\to-\infty}\Lambda_j(\Theta,\rho/T(\Theta)^4)\,T(\Theta)^2 \geq 1 + \rho
\]

For a corresponding upper bound we reuse the trial function $\psi(\theta)=P_{-1/2}(\cosh\theta)$ which satisfies $-\lapl_{\H^2}\psi=\frac14\psi$. Let $\eta$ be the smooth cutoff supported on the annulus $A(\Theta)=C(\Theta)\setminus C(\Theta+1)$, with $\eta\equiv 1$ on $C(\Theta+1)$, constructed above (recall that $\Theta<0$). Define $f=\eta\psi\in H_0^2(C(\Theta))$.
Using \eqref{eq:vibTTgradratio} and \eqref{eq:vibTTlaplratio} (applied to this same cutoff $f$) we have
\[
\lim_{\Theta \to -\infty} \frac{\int_{C(\Theta)}(\lapl f)^2\,dV}{\int_{C(\Theta)}|\grad f|^2\,dV} = \frac14,
\qquad
\lim_{\Theta \to -\infty}\frac{\int_{C(\Theta)}f^2\,dV}{\int_{C(\Theta)}|\grad f|^2\,dV}=4,
\]
as $\Theta\to-\infty$. Therefore,
\[
\limsup_{\Theta\to-\infty}\Lambda_1(\Theta,\rho/T(\Theta)^4)\,T(\Theta)^2
\le \lim_{\Theta\to-\infty}\left(\frac{T(\Theta)^2}{4}+\frac{4\rho}{T(\Theta)^2}\right)=1+\rho.
\]
As in the vibrating case, to extend this estimate to higher eigenvalues we enlarge the disk to contain $j$ disjointly supported translates of $f$, and use their span as a trial subspace. Hence $\limsup\limits_{\Theta \to -\infty} \le 1+\rho$ for each fixed $j$. Combining the lower and upper bounds proves the claim.

\subsubsection*{Spherical $(T,T)$: $\Theta \to \pi$}

After transforming to the unit disk, the Rayleigh quotient becomes
\[
Q_R(v)=
\frac{\displaystyle
\int_{\D}(1+R^2t^2)^2(\lapl v)^2\,dA
+\rho\int_{\D}(1+R^2t^2)^{-2}v^2\,dA}
{\displaystyle\int_{\D}|\grad v|^2\,dA}.
\]
Let $\lambda_1^D(\D)>0$ denote the first Dirichlet Laplacian eigenvalue on $\D$. Since $\rho\leq0$ and $(1+R^2t^2)^{-2}\leq1$, the Poincar\'e inequality gives
\[
\rho\int_{\D}(1+R^2t^2)^{-2}v^2\,dA
\geq \rho\int_{\D}v^2\,dA
\geq \frac{\rho}{\lambda_1^D(\D)}
\int_{\D}|\grad v|^2\,dA.
\]
Consequently,
\[
Q_R(v)\geq
\frac{\displaystyle\int_{\D}(1+R^2t^2)^2(\lapl v)^2\,dA}
{\displaystyle\int_{\D}|\grad v|^2\,dA}
+\frac{\rho}{\lambda_1^D(\D)}.
\]
The first term has eigenvalues tending to $\infty$ by Lemma \ref{lem:ptwisevanfourthorder--numer}, applied with $w_R=(1+R^2t^2)^{-2}$ and $\rho=0$. The second term is a fixed constant, so all eigenvalues tend to $\infty$ as $R\to\infty$.

\subsection*{Buckling limits (Theorem \ref{th:bucklingpos} and Table \ref{tab:buckling-pos-mixed-limits-conj} for $\rho \geq 0$)}

\subsubsection*{Hyperbolic $(S, T)$: $\Theta \to -\infty$}

Now we consider mixed factors and $\rho \geq 0$. We take scaling factor $S(\Theta)$ and normalizing factor $T(\Theta)$ first.

The Rayleigh quotient reads
\begin{equation}\label{eq:idk_theres_a_lot}
\dfrac{\displaystyle \int_{\D}\frac{(1-R^2t^2)^2}{(1-R^2)^2}(\lapl v)^2 \,  dA +  \rho\int_{\D} (1-R^2)^{-2}(1-R^2t^2)^{-2}v^2 \, dA}{\displaystyle \int_{\D} |\grad v|^2 \, dA}.
\end{equation}
We bound the $L^2$ integral below by 0 since $\rho$ is non-negative. Then, the weight $\frac{(1-R^2t^2)^2}{(1-R^2)^2} := w_R^{-1}$ satisfies the hypotheses of Lemma \ref{lem:ptwisevanfourthorder--numer} with $\rho = 0$. Since the weighted eigenvalues in Lemma \ref{lem:ptwisevanfourthorder--numer} are a lower bound, we conclude the mixed-factor eigenvalues in \eqref{eq:idk_theres_a_lot} tend to $\infty$.

\subsubsection*{Spherical $(S,T)$: $\Theta\to\pi$}

The Rayleigh quotient here is
\[
\mathcal Q_R(v):=
\frac{\displaystyle \int_{\D}\frac{(1+R^2t^2)^2}{(1+R^2)^2}(\lapl v)^2\,dA
+\rho\int_{\D}\frac{1}{(1+R^2)^2(1+R^2t^2)^2}v^2\,dA}
{\displaystyle \int_{\D}|\grad v|^2\,dA}.
\]
The proof is similar to the spherical $(S,S)$ case for Theorem \ref{th:vibratingmono}, using the same \(k\)-dimensional trial space
\[
V_k=\mathrm{span}\{\eta_1,\dots,\eta_k\}\subset H_0^2(\D),
\]
with pairwise disjoint supports contained in \(\{t\le R^{-1}\}\). By the scaling estimates from that previous proof, for
\(v=\sum_{i=1}^k c_i\eta_i\),
\[
\int_\D v^2\,dA \sim R^{-2}\sum_{i=1}^k |c_i|^2,\qquad
\int_\D |\grad v|^2\,dA \sim \sum_{i=1}^k |c_i|^2,\qquad
\int_\D (\lapl v)^2\,dA \sim R^{2}\sum_{i=1}^k |c_i|^2.
\]
Moreover, on \(\{t\le R^{-1}\}\),
\[
\frac{(1+R^2t^2)^2}{(1+R^2)^2}\lesssim R^{-4},
\qquad
\frac{1}{(1+R^2)^2(1+R^2t^2)^2}\lesssim R^{-4}.
\]

Therefore
\[
0\le \max_{0\neq v\in V_k}\mathcal Q_R(v)\lesssim R^{-2}+\rho R^{-6}\to 0
\qquad\text{as }R\to\infty
\]
using that $\rho \geq 0$. By the min--max principle, the \(k\)th eigenvalue tends to \(0\) as \(\Theta\to\pi\). Since \(k\) was arbitrary, every eigenvalue tends to \(0\).

\subsubsection*{Hyperbolic $(T,S)$: $\Theta \to -\infty$} 

The variational characterization prior to transformations is
\[
\Lambda_j(\Theta, \rho/S(\Theta)^4)T(\Theta)^2 = \min_{\mathcal{L}}\max_{0 \neq u \in \mathcal{L}}\dfrac{T(\Theta)^2\displaystyle \int_{C(\Theta)} (\lapl u)^2 \, dV +  \frac{T(\Theta)^{2}}{S(\Theta)^4} \rho\int_{C(\Theta)} | u|^2 \, dV} {\displaystyle \int_{C(\Theta)}  |\grad u|^2 \, dV}.
\]
We claim that for each fixed $j\ge 1$,
\[
\lim_{\Theta\to-\infty}\Lambda_j(\Theta,\rho/S(\Theta)^4)\,T(\Theta)^2=1.
\]
The proof here is similar to the hyperbolic $(T,T)$ vibrating and buckling limits. We use Equations \eqref{eq:hypspeclowbound--grad} and \eqref{eq:hypspeclowbound--lapl} and the fact that $\rho \geq 0$ to deduce
\[
\dfrac{T(\Theta)^2\displaystyle \int_{C(\Theta)} (\lapl u)^2 \, dV +  \frac{T(\Theta)^{2}}{S(\Theta)^4} \rho\int_{C(\Theta)} | u|^2 \, dV} {\displaystyle \int_{C(\Theta)}  |\grad u|^2 \, dV} \geq \dfrac{T(\Theta)^2\displaystyle \int_{C(\Theta)} (\lapl u)^2 \, dV} {\displaystyle \int_{C(\Theta)}  |\grad u|^2 \, dV} \geq \frac{T(\Theta)^2}{4}.
\]
Letting $\Theta \to -\infty$ shows $\liminf \geq 1$.

For the upper bound, we reuse the radial function $\psi$ in the hyperbolic $(T,T)$ vibrating and buckling limit proofs. Let $f=\eta\psi\in H_0^2(C(\Theta))$ be defined as it was in those arguments.
Using \eqref{eq:vibTTgradratio} and \eqref{eq:vibTTlaplratio} we have
\[
\frac{\int_{C(\Theta)}(\lapl f)^2\,dV}{\int_{C(\Theta)}|\grad f|^2\,dV}\to \frac14
\qquad\text{as }\Theta\to-\infty.
\]
Now observe that Equation \eqref{eq:hypspeclowbound--grad} implies
\[
0 \leq \dfrac{\frac{T(\Theta)^2}{S(\Theta)^4} \rho \int_{C(\Theta)} |f|^2 \, dV}{\int_{C(\Theta)}|\grad f|^2 \, dV } \leq  4\frac{T(\Theta)^2}{S(\Theta)^4}\rho \to 0.
\]
We conclude that the first eigenvalue tends to $1$. The argument used before of enlarging the disk to contain disjoint copies of $f$ still applies and we conclude the higher eigenvalues tend to $1$ as well.

\subsubsection*{Spherical $(T,S)$: $\Theta \to \pi$}

The Rayleigh quotient here reads

\[
\dfrac{\displaystyle \int_{\D}(1+R^2t^2)^2(\lapl v)^2 \,  dA +  \rho\int_{\D} \frac{(1+R^2)^4}{(1+R^2t^2)^2}v^2 \, dA}{\displaystyle \int_{\D} |\grad v|^2 \, dA}.
\]
As $\rho \geq 0$ and $(1+R^2t^2)^{-2}$ satisfies the hypotheses of Lemma \ref{lem:ptwisevanfourthorder--numer}, we conclude all eigenvalues tend to $\infty$ as $R \to \infty$.

\subsection{Euclidean limits}\label{sec:euc_limits_buck_vibr}

We look at one specific example, as the proof is identical in all cases. Consider the first eigenvalue in Theorem \ref{th:vibratingmono} with variational characterization
\[
\Gamma_j(\Theta, \tau/\si^2(\Theta))\si^4(\Theta) =\min_{\mathcal{L}}\max_{0 \neq v \in \mathcal{L}}\dfrac{\displaystyle \int_{\D}\frac{(1\pm R^2t^2)^2}{(1\pm R^2)^2}(\lapl v)^2 \,  dA +  \tau\int_{\D}|\grad v|^2 \, dA}{\displaystyle \int_{\D} \frac{(1\pm R^2)^2}{(1\pm R^2t^2)^2}v^2 \, dA}
\]
where we take the negative sign in the hyperbolic case and positive sign in the spherical case. It is not hard to see that as $R \to 0$ the coefficients converge to $1$, uniformly with respect to $t$. Interchanging limits with the min-max characterization due to uniform convergence shows that
\[
\Gamma_j(\Theta, \tau/\si^2(\Theta))\si^4(\Theta) \to \Gamma_j(\D, \tau) 
\]
 as $\Theta \to 0$, where $\Gamma_j(\D, \tau)$ is the $j$-th vibrating eigenvalue on the unit disk $\D$ with parameter $\tau$. It is also straightforward to show all the coefficients present in Table \ref{tab:coeff-monotonicity} converge uniformly to 1, and hence all scaled and normalized eigenvalues in Theorems \ref{th:vibratingmono}, \ref{th:bucklingneg}, and \ref{th:bucklingpos} converge to Euclidean eigenvalues on $\D$ as $\Theta \to 0$.

\section*{Funding Acknowledgement}

The author was supported in part by National Science Foundation Award No.~2246537 to the author's advisor, Richard Laugesen. This support is gratefully acknowledged.

\section*{Declaration of generative AI and AI-assisted technologies in the manuscript preparation process}

During the preparation of this manuscript, the author used OpenAI Prism and GPT-5.6 Sol for limited assistance with LaTeX formatting, article organization, language editing, and checking and revising proof details in the spherical buckling limits as $\Theta\to\pi$ for Theorem~\ref{th:bucklingneg}. The underlying research and mathematical results are the author's own. The author reviewed all suggested changes and takes full responsibility for the article.


\begin{thebibliography}{99}

\bibitem{AF04}
C. Amrouche and M. Fontes.
\emph{Biharmonic problem in exterior domains.}
C. R. Math. Acad. Sci. Paris \textbf{338} (2004), no. 2, 121--126.

\bibitem{AF05}
C. Amrouche and M. Fontes.
\emph{Biharmonic problem in exterior domains of ${\mathbb R}^n$: an approach with weighted Sobolev spaces.}
J. Math. Anal. Appl. \textbf{304} (2005), no. 2, 552--571.

\bibitem{C11}
L. M. Chasman.
\emph{An isoperimetric inequality for fundamental tones of free plates.}
Commun. Math. Phys. \textbf{303} (2011), 421--449.

\bibitem{CP22}
B. Colbois and L. Provenzano.
\emph{Neumann eigenvalues of the biharmonic operator on domains: geometric bounds and related results.}
J. Geom. Anal. \textbf{32} (2022), 218.

\bibitem{E10}
L. Evans.
Partial Differential Equations.
Second edition.
Grad. Stud. Math., \textbf{19}.
American Mathematical Society, Providence, RI, 2010.

\bibitem{H25}
S. Harman.
\emph{Scaling inequalities and limits for Robin and Dirichlet eigenvalues.}
J. Math. Anal. Appl. \textbf{544} (2025), no. 2, Paper No. 129082.

\bibitem{H96}
E. Hebey.
Sobolev Spaces on Riemannian Manifolds.
Lecture Notes in Math., \textbf{1635}.
Springer-Verlag, Berlin, 1996.

\bibitem{K93}
B. Kawohl, H.~A. Levine, and W. Velte.
\emph{Buckling eigenvalues for a clamped plate embedded in an elastic medium and related questions.}
SIAM J. Math. Anal. \textbf{24} (1993), no. 2, 327--340.

\bibitem{LL22}
J. J. Langford and R. S. Laugesen.
\emph{Scaling inequalities for spherical and hyperbolic eigenvalues.}
J. Spectr. Theory \textbf{13} (2023), no. 1, 263--296.

\bibitem{L12}
R. S. Laugesen.
Spectral Theory of Partial Differential Equations --- Lecture Notes.
\href{https://arxiv.org/abs/1203.2344}{arXiv:1203.2344} (2012).

\bibitem{M70}
H. P. McKean.
\emph{An upper bound to the spectrum of $\Delta$ on a manifold of negative curvature.}
J. Differential Geom. \textbf{4} (1970), no. 3, 359--366.
\href{https://doi.org/10.4310/jdg/1214429509}{doi:10.4310/jdg/1214429509}.

\bibitem{MM21}
G. Migliaccio and H. A. Matevossian.
\emph{Exterior biharmonic problem with the mixed Steklov and Steklov-type boundary conditions.}
Lobachevskii J. Math. \textbf{42} (2021), 1886--1899.

\bibitem{MM23}
G. Migliaccio and H. A. Matevossian.
\emph{Steklov-Farwig biharmonic problem in exterior domains.}
Lobachevskii J. Math. \textbf{44} (2023), no. 6, 2413--2428.

\bibitem{R77}
R. C. Reilly.
\emph{Applications of the Hessian operator in a Riemannian manifold.}
Indiana Univ. Math. J. \textbf{26} (1977), no. 3, 459--472.
\href{https://doi.org/10.1512/iumj.1977.26.26036}{doi:10.1512/iumj.1977.26.26036}.

\bibitem{DLMF}
NIST Digital Library of Mathematical Functions.
W. J. Olver, A. B. Olde Daalhuis, D. W. Lozier, B. I. Schneider, R. F. Boisvert, C. W. Clark, B. R. Miller, B. V. Saunders, H. S. Cohl, and M. A. McClain, eds.
\url{https://dlmf.nist.gov/}.

\end{thebibliography}
\end{document}